\documentclass[letterpaper]{amsart}
\usepackage{lmodern}
\usepackage[T1]{fontenc}
\usepackage[english]{babel}
\usepackage{amsmath,amscd,amssymb,amsthm,amsxtra}
\usepackage{xfrac}
\usepackage{microtype}
\usepackage{eucal}
\usepackage{mathrsfs}
\usepackage{mathtools}
\usepackage{accents}
\usepackage{enumerate}
\usepackage{color,graphicx,esint}
\usepackage{epsfig,epstopdf}
\usepackage{tikz-cd}
\usetikzlibrary{decorations.pathreplacing,patterns}
\usepackage[notref,notcite]{}
\usepackage{hyperref}
\hypersetup{colorlinks={true}, linkcolor={black}, citecolor={black}}

\theoremstyle{plain}
\newtheorem{theorem}{Theorem}[section]

\newtheorem{lemma}[theorem]{Lemma}

\theoremstyle{definition}
\newtheorem{definition}[theorem]{Definition}
\newtheorem{remark}[theorem]{Remark}

\theoremstyle{remark}
\numberwithin{theorem}{section}
\numberwithin{equation}{section}
\numberwithin{figure}{section}

\def\R{\mathbb{R}}

\def\N{\mathbb{N}}
\def\1{{\bf 1}}
\def\e{\mathbf{e}}
\def\d{\mathrm{d}}
\def\Id{\mathrm{Id}}

\def\XXint#1#2#3{{\setbox0=\hbox{$#1{#2#3}{\int}$}
        \vcenter{\hbox{$#2#3$}}\kern-.5\wd0}}

\DeclareMathOperator{\conv}{conv}

\DeclareMathOperator{\diam}{diam}
\DeclareMathOperator{\dist}{dist}

\DeclareMathOperator{\dom}{dom}

\DeclareMathOperator{\interior}{int}

\DeclareMathOperator{\Span}{span}
\DeclareMathOperator{\spt}{spt}

\newcommand{\mres}{\mathop{\hbox{\vrule height 7pt width .5pt depth 0pt
\vrule height .5pt width 6pt depth 0pt}}\nolimits}

\begin{document}

\title[Regularity Properties of Monotone Measure-preserving Maps]{Regularity Properties of Monotone Measure-preserving Maps}

\author[A.\ Figalli]{Alessio Figalli}
\address{ETH Z\"{u}rich, Department of Mathematics, R\"{a}mistrasse 101, Z\"{u}rich 8092, Switzerland}
\email{alessio.figalli@math.ethz.ch}
\author[Y.\ Jhaveri]{Yash Jhaveri}
\address{Rutgers University\,--\,Newark, Department of Mathematics \& Computer Science, Smith Hall, 101 Warren Street, Newark, New Jersey 07102, USA}
\email{yash.jhaveri@rutgers.edu}

\begin{abstract}
In this note, we extend the regularity theory for monotone measure-preserving maps, also known as optimal transports for the quadratic cost optimal transport problem, to the case when the support of the target measure is an arbitrary convex domain and, on the low-regularity end, between domains carrying certain invariant measures.\\

\begin{center}
{\it In honor of David Jerison for his 70th birthday.}
\end{center}
\end{abstract}
\maketitle


\section{Introduction}

Given a pair of atom-free Borel probability measures $\mu$ and $\nu$ on $\R$, the monotone rearrangement theorem asserts that the function $y = y(x)$ defined implicitly by
\[
\int_{-\infty}^x \, \d \mu = \int_{-\infty}^{y(x)} \, \d \nu
\]
is measure-preserving, i.e.,
\[
\mu(y^{-1}(E)) = \nu(E)\quad \text{for any Borel set} \quad E \subset \R.
\]
In particular, $y$ is unique $\mu$ almost everywhere and can be made to be monotone, or, equivalently, the derivative of a convex function.

When $\R$ is replaced by $\R^n$, however, before Brenier's discovery (see \cite{B}), a proper generalization of the monotone rearrangement theorem was myth; and only after the work of McCann, in \cite{M}, was the myth made real.
Precisely, he proved that if $\mu$ vanishes on every Lipschitz $(n-1)$-dimensional surface\footnote{\,McCann actually assumes that $\mu$ vanishes on all (Borel) sets of Hausdorff dimension $n-1$.
This guarantees that the set of non-differentiability points of a convex function are $\mu$-negligible. 
However, this assumption can be weakened. 
Since convex functions are differentiable outside of a countable union of Lipschitz hypersurfaces (see \cite{Z}), McCann's theorem holds assuming that $\mu$ vanishes on every Lipschitz $(n-1)$-dimensional surface.}, then a convex potential $u : \R^n \to \R \cup \{ + \infty \}$ exists whose gradient map $\nabla u = \nabla u(x)$ is unique $\mu$ almost everywhere and pushes $\mu$ forward to $\nu$, i.e.,
\[
\mu((\nabla u) ^{-1}(E)) = \nu(E) \quad \text{for any Borel set} \quad E \subset \R^n.
\]
(Brenier's theorem guaranteed the same conclusion as McCann's theorem, but under some restrictive technical conditions on $\mu$ and $\nu$.)

The first general regularity result on these Brenier--McCann maps was proved by Caffarelli in \cite{C2}: provided that $\mu$ and $\nu$ are absolutely continuous with respect to $n$-dimensional Lebesgue measure $\mathscr{L}^n$, their respective densities $f$ and $g$ vanish outside of and are bounded away from zero and infinity on open bounded sets $X$ and $Y$ respectively, and $Y$ is convex, he showed that $u$ is strictly convex in $X$ (see, e.g., the proof of \cite[Theorem 4.6.2]{F}).
This result opened the door to the development of a regularity theory for mappings with convex potentials based on the regularity theory for strictly convex solutions to the Monge--Amp\`{e}re equation (see, e.g., \cite{C0}); indeed,
\[
(\nabla u)_{\#} f = g \text{ is formally, at least, equivalent to }  \det D^2 u = \frac{f}{g(\nabla u)}.
\]

Unfortunately, Caffarelli's boundedness assumptions on the domains $X$ and $Y$ are restrictive, since many probability densities, especially those found in applications, are supported on all of $\R^n$: Gaussian densities, for example.
Motivated by this, in \cite{FCE}, Cordero-Erausquin and Figalli showed that, in several situations of interest, one can ensure the regularity of monotone measure-preserving maps even if the measures under consideration have unbounded supports. 
However, missing from their collection is the situation where $Y$ is an {\it arbitrary} convex domain.
Lifting this restriction is a main goal of this paper.

\subsection{Results}
Our main theorem is an extension of Caffarelli's theorem, on the strict convexity of $u$, in two ways.
First, we allow $X$ and $Y$ to be unbounded.
Second, we permit $X$ and $Y$ to carrying certain invariant measures that we call {\it locally doubling measures} (qualitatively, our notion replaces balls in the classical notion of a doubling measure with ellipsoids, in order to account for the affine invariance of our setting).

\begin{definition}
A nonnegative measure $\lambda$ is {\it locally doubling (on ellipsoids)} if the following holds: for every ball $B$, there is a constant $C \geq 1$ such that  
\[
\lambda(\mathcal{E}) \leq C \lambda(\tfrac{1}{2}\mathcal{E})
\]
for all ellipsoids $\mathcal{E} \subset B$ with center (of mass) in $\spt(\lambda)$.
Here $\frac{1}{2}\mathcal{E}$ is the dilation of $\mathcal{E}$ with respect to its center by $1/2$.
\end{definition}
This notion of doubling was introduced by Jhaveri and Savin in \cite{JS}.\footnote{\,This family of measures is strictly larger than the family of measures locally comparable to Lebesgue measure on their supports. 
(See \cite{JS} for examples of locally doubling measures not comparable to Lebesgue measure on their supports.)}
That said, the first consideration of measures with a ``doubling like'' property in the world of solutions to Monge--Amp\`{e}re equations can be traced back to the work of Jerison \cite{J} and then Caffarelli \cite{C1}. 
In particular, in \cite{C1}, Caffarelli showed that Alexandrov solutions to
\[
\det D^2 v = \rho,
\]
where the measure $\rho$ is doubling on a specific collection of convex sets called sections\footnote{\,These are sets of the form $\{ v \leq \ell \}$ for any affine function $\ell$.}, share the same geometric properties as Alexandrov solutions to Monge--Amp\`{e}re equations with right-hand sides comparable to Lebesgue measure (see \cite{C0}).

We now state our main theorem.

\begin{theorem}
\label{thm: strict convexity}
Let $\mu$ and $\nu$ be two locally doubling probability measures on $\R^n$ that vanish on Lipschitz $(n-1)$-dimensional surfaces and are concentrated on two open sets $X$ and $Y$ respectively,
and suppose that $Y$ is convex.
Then any convex potential $u$ associated to the Brenier--McCann map pushing $\mu$ forward to $\nu$ is strictly convex in $X$.
\end{theorem}

\begin{remark}
It is well-known that $Y$ needs to be convex.
When $Y$ is not convex, $u$ can fail to be strictly convex and $\nabla u$ can behave rather poorly.
If we consider Pogorelov's counterexample to the strict convexity of solutions to the Monge--Amp\`{e}re equation in three dimensions, we see that $Y$ needs to be convex in order to guarantee the strict convexity of potentials of Brenier--McCann maps.
In particular, let $u(x',x_3) = |x'|^{4/3}(1+x_3^2)$, let $Q_r = \{ |x_i| < r/2 \text{ for } i = 1,2,3 \}$ be the cube with side length $r > 0$ and centered at the origin in $\R^3$, and let $Y = \nabla u(Q_r)$.
If $r \ll 1$, then $f = \det D^2 u$ is analytic and positive in $Q_r$, and $\nabla u$ is the Brenier--McCann map pushing forward $\mu = (f/\|f\|_{L^1(Q_r)}) \mathscr{L}^3 \mres Q_r$ to $\nu = (1/\mathscr{L}^3(Y)) \mathscr{L}^3 \mres Y$.
The set $Y$ is open but not convex, and $u = 0$ along $\{ x' = 0 \}$.
Moreover, as demonstrated in, for instance, \cite{C2} and \cite{Jh}, $\nabla u$ easily fails to be continuous when $Y$ is not convex.
\end{remark}

\begin{remark}
While the target measure $\nu$ need not vanishes on all Lipschitz $(n - 1)$-dimensional surfaces in order to invoke McCann's theorem (which asks this only of the source measure $\mu$), it must in order to ensure our main theorem holds.
If $\mu = \mathscr{L}^2 \mres Q_1$ is the uniform measure on $Q_1$ the unit cube centered at the origin in $\R^2$ and $\nu = \mathscr{H}^1 \mres Q_1 \cap \{ x_2 = 0 \}$ is the $1$-dimensional Hausdorff measure restricted to the central horizontal axis of $Q_1$, then the Brenier--McCann map pushing $\mu$ forward to $\nu$ is the projection map $(x_1,x_2) \mapsto x_1$.
Up to a constant, this map's convex potential is $\frac{1}{2}x_1^2$, which is not strictly convex.
With respect our proof of Theorem~\ref{thm: strict convexity}, asking this of both $\mu$ and $\nu$ guarantees the validity of the mass balance formula in Lemma~\ref{lem: Brenier soln} (see also Remark~\ref{rmk:balance}), our main tool.
\end{remark}

\begin{remark}
A simple case to which Theorem~\ref{thm: strict convexity} applies, but the corresponding results in \cite{C2} and \cite{FCE} do not, is when $\mu = g \mathscr{L}^3 \mres \{|x_1| < 1 \} \times \R^2$ and $\nu = g \mathscr{L}^3 \mres \R^2 \times \{|x_3| < 1 \}$, and $g$ is the standard Gaussian density on $\R^3$ appropriately normalized to make $\mu$ and $\nu$ probability measures.
\end{remark}

With our main theorem in hand, our second and third theorems further extend the known regularity theory for monotone measure-preserving maps, completing the story started by Cordero-Erausquin and Figalli in \cite{FCE} on monotone transports between unbounded domains.

\begin{theorem}
\label{thm: C1}
Let $\mu$ and $\nu$ be two locally doubling probability measures on $\R^n$ that vanish on Lipschitz $(n-1)$-dimensional surfaces and are concentrated on two open sets $X$ and $Y$ respectively, and suppose that $Y$ is convex.
Then the Brenier--McCann map $\nabla u$ pushing $\mu$ forward to $\nu$ is a homeomorphism from $X$ onto a full measure subset of $Y$.
Moreover, for every $A \Subset X$, a constant $\alpha > 0$ exists such that $\nabla u \in C^{0,\alpha}(A)$.
Furthermore, $\nabla u(X) = Y$ whenever $X$ is convex.
\end{theorem}

\begin{theorem}
\label{thm: higher reg}
Let $f$ and $g$ be two functions on $\R^n$ that define locally doubling probability measures concentrated on two open sets $X$ and $Y$ respectively, and suppose that $Y$ is convex.
Assume that $f$ and $g$ are bounded away from zero and infinity on compact subsets of $X$ and $Y$ respectively. 
Then for every $E \Subset X$, a constant $\epsilon > 0$ exists such that any convex potential $u$ associated to the Brenier--McCann map pushing $f$ forward to $g$ is $W^{2,1+\epsilon}(E)$.
Also, $\nabla u$ is locally a $C^{k+1,\beta}$-diffeomorphism from $X$ onto its image provided $f$ and $g$ are locally $C^{k,\beta}$ in $X$ and $Y$ respectively.
\end{theorem}

\begin{remark}
We note that the proof of the Theorem~\ref{thm: higher reg}, given the strict convexity of $u$ (provided by Theorem~\ref{thm: strict convexity}), is classical.
Indeed, it suffices to localize  classical regularity results for the Monge--Amp\`ere equation.
We refer the reader to \cite[Section 4.6.1]{F} for more details.
\end{remark}

\subsection{Structure}
This remainder of this paper is structured as follows.

In Section~\ref{sec: strict convexity}, we prove Theorem~\ref{thm: strict convexity}.
Our proof is self-contained apart from some facts in convex analysis; we provide explicit references to these used but unproved facts.
We remark that our proof is inspired by the proof of the Alexandrov maximum principle in \cite{JS} (and, of course, Caffarelli's original proof of the strict convexity of potential functions of optimal transports/solutions to Monge--Amp\`{e}re equations).
If the reader is familiar with \cite{FCE} or \cite{C2}, then they might consider directing their attention to Case 2.
Case 2b is completely novel. 
Case 2a illustrates our argument in the setting of \cite{C2}, which builds on the work of \cite{C0} and is the foundation for Case 2b.

Section~\ref{sec: C1} is dedicated to the proof of Theorem~\ref{thm: C1}.
Our proof here is similarly self-contained (and an adaptation of Caffarelli's argument of the same result in \cite{C0}, but, of course, using the line of reasoning developed to prove Theorem~\ref{thm: strict convexity}).
The H\"{o}lder regularity of $\nabla u$ is a consequence of appropriately localizing the arguments of \cite{JS}.

\section{Proof of Theorem~\ref{thm: strict convexity}}
\label{sec: strict convexity}

Before we begin, it will be convenient to replace the potential $u$ by the following lower-semicontinuous extension of $u$ outside of $X$:\footnote{\,Here $\partial u(z)$ is called the subdifferential of $u$ at $z$ and is defined as follows:
\[
\partial u(z) := \{ p \in \R^n : u(x) \geq u(z) + p \cdot (x-z) \text{ for all $x \in X$} \}.
\]
Moreover, for a set $E \subset \R^n$, we define $\partial u(E) := \cup_{z \in E} \partial u(z)$.
}
\[
\underline{u}(x) := \sup_{\substack{z \in X \\ p \in \partial u(z)}} \{ u(z) + p \cdot (x-z) \}.
\]
Observe that $\underline{u}|_X = u|_X$.
For notational simplicity, we shall not distinguish $\underline{u}$ from $u$; so when we write $u$ in what follows, we mean $\underline{u}$.

We shall denote the domain of $u$ by $\dom(u)$, namely, $\dom(u) := \{u<+\infty\}$.
Note that $\dom(u)$ is convex.
We recall that convex functions are locally Lipschitz inside their domain (see, e.g., \cite[Appendix A.4]{F}).
Furthermore, we shall denote the convex hull of a set $A$ by $\conv(A)$.

Let $\ell$ define a supporting plane to the graph of $u$ at a point in $X$.
Precisely,
\[
\ell(x)=u(z) + p \cdot (x-z)\quad \text{for some}\quad (z,p) \in X \times \R^n
\]
and $\ell \leq u$.
Note that 
\[
\Sigma := \{ u = \ell \} = \{ u \leq \ell \}
\]
is closed, as $u$ is lower-semicontinuous, 
\[
\Sigma, X \subset \dom(u),
\]
and, because $Y$ is convex,
\begin{equation}
\label{eqn: subdiff props}
\partial u(\R^n) \subset \overline{\partial u(X)} \subset \overline{Y}\quad \text{and}\quad \mathscr{L}^n(Y \setminus \partial u(X)) = 0. 
\end{equation}
(A proof of \eqref{eqn: subdiff props} can be found in \cite{FCE}.)

Recall that an exposed point $\hat{x}$ of $\Sigma\subset \R^n$ is one for which there exists a hyperplane $\Pi\subset \R^n$ tangent to $\Sigma$ at $\hat{x}$ such that $\Pi\cap \Sigma=\{\hat{x}\}$. 
Also, remember that optimal/monotone transports balance mass, in the following way.

\begin{lemma}[Mass Balance Formula]
\label{lem: Brenier soln}
Let $u : \R^n \to \R \cup \{ + \infty \}$ be convex and such that $(\nabla u)_\# \mu = \nu$ where $\mu$ and $\nu$ are two Borel measures that vanish on all Lipschitz $(n-1)$-dimensional surfaces.
Then for all Borel sets $E \subset \R^n$,
\[
\mu(E)  = \nu (\partial u(E) ).
\]
\end{lemma}

\begin{remark}
\label{rmk:balance}
The mass balance formula was originally proved for measures that are absolutely continuous with respect to Lebesgue measure (see, e.g., \cite[Lemma 4.6]{V}).
However, with respect to absolute continuity, the proof only relies on the measures in question not giving mass to the set of non-differentiable points of a convex function.
As observed in the introduction, such points are contained in a countable union of Lipschitz $(n-1)$-dimensional surfaces. 
So the set of non-differentiable points of a convex function is negligible both for $\mu$ and $\nu$ under our assumption.
\end{remark}

Finally, recall that if a nonnegative measure is locally doubling (on ellipsoids), then it is locally doubling on all bounded convex domains (see \cite[Corollary 2.5]{JS}).
After this preliminary discussion, we can now prove our main theorem.

\begin{proof}[Proof of Theorem~\ref{thm: strict convexity}]
Proving that $u$ is strictly convex in $X$ corresponds to proving that for any supporting plane $\ell$ to (the graph of) $u$ at a point in $X$, the set $\Sigma=\{u=\ell\}$ is a singleton.
Assuming that $\Sigma$ is not a singleton, we will show that $\Sigma$ both has and does not have exposed points, which cannot be; thus, $\Sigma$ is a singleton, as desired.

\subsection*{Case 1}{\bf $\Sigma$ has no exposed points.} If $\Sigma$ has no exposed points, then $\Sigma \supset \R\e$ for some unit vector $\e$.
In turn, $\partial u(\R^n) \subset \e^\perp$  (see, e.g., \cite[Lemma A.25]{F}).
But this is impossible given \eqref{eqn: subdiff props}:
\[
0 < \mathscr{L}^n(\partial u(X) \cap Y) \leq \mathscr{L}^n(\partial u(\R^n) \cap Y) \leq \mathscr{L}^n(\e^\perp \cap Y) = 0.
\]

\subsection*{Case 2}{\bf $\Sigma$ has an exposed point $\hat{x}$ in $\overline{X}$.}
Up to a translation and a rotation, we can assume that
\[
\hat{x} = 0 \in \overline{X},\quad \Sigma \subset \{ x_1 \leq 0 \},\quad \text{and}\quad  \Sigma \cap \{ x_1 = 0 \} = \{ 0 \}.
\]
Since $X$ is open and $\Sigma\cap X$ is nonempty by construction, there is a point $x_{int}\in \Sigma \cap X$ and a ball centered at this point completely contained in $X$.
Thus, up to a shearing transformation $x \mapsto x - \eta x_1$ with $\eta \cdot \e_1 = 0$, and a dilation, we may assume that
\[
x_{int} = - \e_1 \quad \text{and}\quad
B_d(- \e_1 ) \Subset X
\]
for some $d > 0$.
Finally, up to subtracting $\ell$ from $u$, we can assume that 
\[
\ell \equiv 0.
\]

\subsection*{Case 2a: $0 \in \interior(\dom(u)) \cap \overline{X}$.} 
As our exposed point $0$ and all of the points in $\overline{B_d(-\e_1)}$ belong to $\interior(\dom(u))$, which is convex (and, by definition, open), the convex hull of the union of $\overline{B_d(-\e_1)}$ and $\{0\}$ is contained in $\interior(\dom(u))$.
So there exists an open, bounded set $U \Subset \interior (\dom (u))$ containing $\conv (\overline{B_d(-\e_1)} \cup \{0\} )$.
Moreover, we know that 
\[
\partial u(U) \subset \conv(\overline{\nabla u(U)}) =: \Upsilon  \subset B_R \cap \overline{Y}
\]
for some $R > 0$ (see, e.g., \cite[Lemma A.22]{F}).
Let $u^\ast$ be the Legendre transform of $u$, namely
\begin{equation}
\label{eq:u ast}
u^\ast(q):=\sup_{x \in \R^n}\{q\cdot x-u(x)\}
\end{equation}
and define
\[
\Omega := \partial u^\ast(\Upsilon) \supset U.
\]
Recalling that $\partial u$ and $\partial u^\ast$ are inverses of each other (see, e.g., \cite[Section A.4.2]{F}), 
we deduce that $(\nabla u)_{\#} \rho = \gamma$ where
\[
\rho := \mu \mres \Omega\quad \text{and} \quad \gamma := \nu \mres \Upsilon.
\]
In particular, if we let $\phi : \R^n \to \R \cup \{ + \infty \}$ be defined by
\[
\phi(x) := \sup_{\substack{z \in \Omega \\ p \in \partial u(z)}} \{ u(z) + p \cdot (x-z) \},
\]
then, by construction, $\phi$ and $u$ agree on $\Omega$, $\phi$ is (globally) Lipschitz,
\[
\partial \phi(\R^n) = \partial \phi (\Omega) = \Upsilon,
\] 
and $0 \in \Omega$ is an exposed point for $\{ \phi = 0 \} = \{ \phi \leq 0 \}$.

Now, let $ \phi_\epsilon(x) := \phi(x) - \epsilon(x_1 + 1)$ and define
\[
S_0 := \{ \phi = 0 \} \cap \{ x_1 \geq -1 \} \quad \text{and}\quad S_\epsilon := \{ \phi_\epsilon \leq 0 \}.
\]
Also, let $\gamma_\epsilon$ be defined by
\[
\gamma_\epsilon :=   (\Id - \epsilon \e_1)_{\#} \gamma .
\]
Notice that, by construction, $S_0$ is compact,  $S_0\subset\{x_1\leq 0\}$, $0,-\e_1 \in S_0$, and $S_\epsilon \to S_0$ in the Hausdorff sense as $\epsilon \to 0$; in particular,
there exists $D>0$ such that
\[
S_\epsilon \subset B_D\quad \text{for all}\quad \epsilon \ll 1.
\]
Also, if $a_\epsilon > 0$ is such that $\Pi_\epsilon := \{ x_1 = a_\epsilon \}$ is a supporting plane to $S_\epsilon$, we see that 
\[
a_\epsilon \to 0 \quad\text{as}\quad \epsilon \to 0
\]
and
\[
\epsilon = |\phi_\epsilon(0)| \leq \max_{S_\epsilon} |\phi_\epsilon| \leq (1 + a_\epsilon) \epsilon.
\]

Let $A_\epsilon$ be the John transformation (affine map) that normalizes $S_\epsilon$ (see, e.g., \cite{Gu}): 
\[
A_\epsilon x := L_\epsilon(x - x_\epsilon),
\]
where $x_\epsilon$ is the center of mass of $S_\epsilon$ and $L_\epsilon:\R^n\to\R^n$ is a symmetric and positive definite linear transformation.
Set
\[
\tilde{\phi}_\epsilon(x) := \frac{\phi_\epsilon(A_\epsilon^{-1}x)}{\epsilon} \quad\text{and}\quad \tilde{S}_\epsilon := A_\epsilon(S_\epsilon).
\]
Then
\[
B_1 \subset \tilde{S}_\epsilon \subset B_{n^{3/2}}
\]
and
\[
1 = |\tilde{\phi}_\epsilon(\tilde{0}_\epsilon)| \leq \max_{\tilde{S}_\epsilon} |\tilde{\phi}_\epsilon| \leq 1 + a_\epsilon \quad \text{with}\quad \tilde{0}_\epsilon := A_\epsilon(0).
\]
Recall that affine transformations preserve the ratio of the distances between parallel planes; therefore, letting $\Pi_{-1} := \{ x_1 = -1 \}$, $\Pi_0 := \{ x_1 = 0 \}$, and $\tilde{\Pi}_i := A_\epsilon(\Pi_i)$ for $i = -1, 0, \epsilon$, we have that
\[
\frac{\dist(\tilde{\Pi}_0,\tilde{\Pi}_\epsilon)}{\dist(\tilde{\Pi}_{-1},\tilde{\Pi}_\epsilon)} =  \frac{\dist(\Pi_0,\Pi_\epsilon)}{\dist(\Pi_{-1},\Pi_\epsilon)} = \frac{a_\epsilon}{1+a_\epsilon}.
\]
In turn,
\[
\dist(\tilde{0}_\epsilon, \partial \tilde{S}_\epsilon) \leq \dist(\tilde{\Pi}_0,\tilde{\Pi}_\epsilon) \leq  \dist(\tilde{\Pi}_{-1},\tilde{\Pi}_\epsilon)\frac{a_\epsilon}{1+a_\epsilon} \leq \diam(\tilde{S}_\epsilon)a_\epsilon \leq 2n^{3/2}a_\epsilon,
\]
and considering the cone generated by $\partial \tilde{S}_\epsilon$ over $(\tilde{0}_\epsilon,\tilde{\phi}_\epsilon(\tilde{0}_\epsilon))$, we find that 
\[
K_\epsilon := \conv \left( B_{r_n} \cup \left\{ \tfrac{r_n}{a_\epsilon}\e_1 \right\} \right) \subset \partial \tilde{\phi}_\epsilon(\tilde{S}_\epsilon) \quad \text{with}\quad r_n := \frac{1}{2n^{3/2}}.
\]
(For more details on this inclusion, see, e.g., \cite[Theorem 2.8]{F}.)
So if we let 
\[
\tilde{\rho}_\epsilon := (A_\epsilon)_{\#} \rho \quad \text{and}\quad \tilde{\gamma}_\epsilon := (\epsilon^{-1}L_\epsilon^{-1})_{\#}\gamma_\epsilon,
\]
then $(\nabla \tilde{\phi}_\epsilon)_\# \tilde{\rho}_\epsilon  = \tilde{\gamma}_\epsilon$, and, by the mass balance formula,
\begin{equation}
\label{eqn: mb1}
\tilde{\gamma}_\epsilon (K_\epsilon)  \leq \tilde{\gamma}_\epsilon\big(\partial \tilde{\phi}_\epsilon(\tilde{S}_\epsilon)\big) = \tilde{\rho}_\epsilon (\tilde{S}_\epsilon)\leq \tilde{\rho}_\epsilon(B_{n^{3/2}}).
\end{equation}
On the other hand, since $S_\epsilon \subset B_D$ and $\tilde  S_\epsilon \supset B_1$ for $\epsilon \ll 1$, we see that 
\[
|A_\epsilon(w) - A_\epsilon(z)| \geq \frac{1}{D}|w - z| \quad \text{for all}\quad w,z \in \R^n.
\]
In turn, for $\epsilon \ll 1$,
\[
\tilde{\Omega}_\epsilon := A_\epsilon(\Omega) \supset A_\epsilon(B_d(-\e_1)) \supset B_{\frac{d}{D}}(A_\epsilon(-\e_1)).
\]
Therefore, if we define
\[
\tilde{S}_{\epsilon,d} := \left\{ \dist(\,\cdot\,,\partial \tilde{S}_\epsilon) \geq \tfrac{d}{2D} \right\},
\] 
then there exists a dimensional constant $C_n > 0$  and a point $\tilde{z}_d$ such that
\[
A_\epsilon(B_d(-\e_1)) \cap \tilde{S}_{\epsilon,d} \supset B_{\frac{d}{C_nD}}(\tilde{z}_d).
\]
Also, by, for example, \cite[Corollary A.23]{F},
\[
\partial \tilde{\phi}_\epsilon (\tilde{S}_{\epsilon,d}) \subset B_{\frac{6D}{d}} .
\]
Thus, for all $\epsilon \ll 1$,
\begin{equation}
\label{eqn: mb2}
\begin{split}
\tilde{\rho}_\epsilon (B_{n^{3/2}})  \leq \tilde{\mu}_\epsilon (B_{2n^{3/2}}(\tilde{z}_d))
\leq C_\mu^k \tilde{\mu}_\epsilon\left(B_{\frac{d}{C_n D}}(\tilde{z}_d)\right) = C_\mu^k \tilde{\rho}_\epsilon\left(B_{\frac{d}{C_n D}}(\tilde{z}_d)\right),
\end{split}
\end{equation}
where $\tilde{\mu}_\epsilon := (A_\epsilon)_{\#}\mu$, the number $k \in \N$ is such that $2n^{3/2} \leq 2^k \frac{d}{C_nD}$, and $C_\mu$ is the doubling constant for $\mu$ in $B_{4Dn^{3/2}}$.
(The last equality holds since $\tilde{\mu}_\epsilon$ and $\tilde{\rho}_\epsilon$ agree on $\tilde{\Omega}_\epsilon$.)
Moreover, using the mass balance formula again, we deduce that
\begin{equation}
\label{eqn: mb3}
\begin{split}
\tilde{\rho}_\epsilon\left(B_{\frac{d}{C_n D}}(\tilde{z}_d)\right)  &\leq \tilde{\rho}_\epsilon\left(A_\epsilon(B_d(-\e_1)) \cap \tilde{S}_{\epsilon,d}\right) 
\\
&= \tilde{\gamma}_\epsilon\left(\partial \tilde{\phi}_\epsilon (A_\epsilon(B_d(-\e_1)) \cap \tilde{S}_{\epsilon,d})\right) \leq \tilde{\gamma}_\epsilon\big(B_{\frac{6D}{d}}\big)
\end{split}
\end{equation}
for all $\epsilon \ll 1$.
Consequently, combining the three chains of inequalities \eqref{eqn: mb1}, \eqref{eqn: mb2}, and \eqref{eqn: mb3}, we have that
\begin{equation}
\label{eqn: mb4}
\tilde{\gamma}_\epsilon(K_\epsilon) 
\leq C_\mu^k\tilde{\gamma}_\epsilon\big(B_{\frac{6D}{d}}\big).
\end{equation}
Now, let $t_m \e_1 \in K_\epsilon$ for $m = 1, \dots, M$ be a sequence of points chosen\footnote{\,A possible way to construct such a sequence is to choose $t_m=5^m$. 
To ensure that $t_m \e_1 \in K_\epsilon$ for any $m=1,\ldots,M$, one needs $M \leq \frac{\log r_n - \log a_\epsilon}{\log 5}$.} so that 
\[
\frac{1}{2}K_m \subset K_m \setminus K_{m-1} \quad \text{with}\quad K_m := \conv ( B_{r_n}  \cup  \{t_m\e_1\} ) \quad\text{and}\quad K_0 := B_{r_n}. 
\]
By construction, $\{ \frac{1}{2}K_m \}_{m = 1}^M$ is a disjoint family, and
\begin{equation}
\label{eq:M}
M = M(a_\epsilon) \to \infty\quad \text{as}\quad a_\epsilon \to 0.
\end{equation}
Hence, since $\epsilon L_\epsilon(K_\epsilon) \subset \Upsilon$, we find that
\[
M\tilde{\gamma}_\epsilon(B_{r_n}) \leq \sum_{m=1}^M \tilde{\gamma}_\epsilon(K_m) \leq C_\gamma \sum_{m=1}^M \tilde{\gamma}_\epsilon(\tfrac{1}{2} K_m) \leq C_\gamma \tilde{\gamma}_\epsilon(K_\epsilon),
\]
with $C_\gamma$ denoting the doubling constant for $\gamma$ in $B_{2R}$, which is the same as the doubling constant for $\nu$ in $B_{2R}$; since $\Upsilon$ is convex, $\gamma$ inherits its doubling property from $\nu$.
All in all, considering the above chain of inequalities and \eqref{eqn: mb4}, and denoting by $j\in \mathbb N$ the smallest number such that $\frac{6D}{d} \leq 2^j r_n$, we see that
\[
0 < M\tilde{\gamma}_\epsilon(B_{r_n}) \leq C_\mu^kC_\gamma \tilde{\gamma}_\epsilon\big(B_{\frac{6D}{d}}\big)
\leq C_\mu^kC_\gamma^{j+1}\tilde{\gamma}_\epsilon(B_{r_n}),
\]
or, equivalently,
\[
M \leq C_\mu^kC_\gamma^{j+1}.
\]
But this is impossible for small $\epsilon$, concluding the proof.

\subsection*{Case 2b: $0 \in \partial (\dom(u)) \cap \overline{X}$.}
In this subcase, let $u_\epsilon(x) := u(x) - \epsilon(x_1 + 1)$ and define 
\[
S_0 := \Sigma \cap \{ x_1 \geq -1 \} \quad \text{and} \quad S_\epsilon := \{ u_\epsilon \leq 0 \}.
\]
Like before, for all $\epsilon \ll 1$,
\[
S_\epsilon \subset B_D
\]
for some $D > 0$. Here, however, as $0 \in \partial (\dom(u))$, we have that
\[
\partial \tilde{u}_\epsilon(\tilde{S}_\epsilon) \supset \conv (B_{r_n} \cup \R^+\e_1 ) \quad \text{with}\quad r_n := \frac{1}{2n^{3/2}}.
\]
The function $\tilde{u}_\epsilon$ is defined in an analogous fashion to how $\tilde{\phi}_\epsilon$ was defined in Case 2a (but replacing $\phi$ by $u$) and, again, $\tilde{S}_\epsilon := A_\epsilon(S_\epsilon)$ with $A_\epsilon$ denoting the John map associated to $S_\epsilon$ whose linear part is $L_\epsilon$.
In turn, arguing as we did in Case 2a, where again $k \in \N$ is such that $2n^{3/2} \leq 2^k \frac{d}{C_nD}$ and $C_\mu$ is the doubling constant for $\mu$ in $B_{4Dn^{3/2}}$, but in the original variables, we deduce that
\[
\nu_\epsilon\big(\epsilon L_\epsilon (K_\epsilon)\big) \leq C_\mu^k \nu_\epsilon\big(\epsilon L_\epsilon (B_{\frac{6D}{d}} )\big)
\]
for all $\epsilon \ll 1$ (cf. \eqref{eqn: mb4}).
Here, instead, $\nu_\epsilon := (\Id - \epsilon \e_1)_{\#} \nu$ and
\[
K_\epsilon := \conv \left( B_{r_n} \cup \left\{ \tfrac{r_n}{\epsilon}\e_1 \right\} \right).
\]
Now notice that 
\[
|L_\epsilon(\e_1)| = |A_\epsilon(0) - A_\epsilon(-\e_1)| \leq 2n^{3/2}.
\] 
Moreover, we claim there exists an $N \gg 2n^{3/2} > 0$ such that 
\[
\|\epsilon L_\epsilon\| \leq N \quad \text{for all}\quad \epsilon \ll 1. 
\]
Indeed, if not, then we can find a sequence of points $z_\epsilon \in S_\epsilon$ and slopes $p_\epsilon \in \partial u_\epsilon(z_\epsilon) \cap \Span (S_0)^\perp$ such that $|p_\epsilon| \to \infty$.
In particular, in the limit, we find a point $z_0 \in S_0$ such that $\partial u(z_0) \cap \Span (S_0)^\perp$ contains a sequence of slopes $\{p_j\}_{j \in \N}$ with $|p_j| = j$.
But as $p_j \in \Span (S_0)^\perp$, we see that $p_j \cdot (x - z_0) =  p_j \cdot x  =   p_j \cdot (x - z)$ for any $z \in S_0$.
Hence, $p_j \in \partial u(z)$ for all $z \in S_0$ and $j \in \N$.
However, this is impossible; $S_0 \cap \interior(\dom(u))$ is nonempty, and on this set, $u$ is locally Lipschitz, proving the claim.

Therefore, 
\[
\epsilon L_\epsilon(K_\epsilon) \subset B_N \quad \text{and}\quad \epsilon L_\epsilon \big(B_{\frac{6D}{d}}\big) \subset B_{\frac{6DN}{d}}.
\]
And so, arguing exactly like we did in Case 2a, we find that
\[
M \leq C_\mu^kC_\nu^{j+1},
\]
where $C_\nu$ is the doubling constant for $\nu$ in $B_{6DN/d}$ and $M = M(\epsilon) \to \infty$ as $\epsilon \to 0$ is the analogous count for this case's $K_\epsilon$ (cf. \eqref{eq:M}).
But, again, this is impossible.

\subsection*{Case 3}{\bf $\Sigma$ has an exposed point $\hat{x}$ in $\R^n \setminus \overline{X}$.}
In this case, up to a translation, a dilation, a rotation, and subtracting  $\ell$ from $u$, we can assume that
\[
\hat{x} = 0,\quad \Sigma \subset \{ x_1 \leq 0 \},\quad \ell \equiv 0,\quad \text{and}\quad S_0 := \Sigma \cap \{ x_1 \geq - 1 \} \Subset \R^n \setminus \overline{X}.
\]
Like before, let $u_\epsilon(x) := u(x) - \epsilon(x_1 + 1)$ and define
\[
S_\epsilon := \{ u_\epsilon \leq 0 \} \quad \text{and}\quad \nu_\epsilon := (\Id - \epsilon \e_1)_{\#} \nu.
\]
Again, $S_\epsilon \to S_0$ as $\epsilon \to 0$, so
\[
\diam(S_\epsilon) \leq 2\diam(S_0) \quad\text{and}\quad S_\epsilon \Subset \R^n \setminus \overline{X}
\] 
for all $\epsilon \ll 1$.
For these small positive $\epsilon$, then,
\[
0 = \mu(S_\epsilon) = \nu_\epsilon(\partial u_\epsilon(S_\epsilon)),
\]
where the second equality follows from the mass balance formula.
(Recall that $\mu$ vanishes on $\R^n \setminus \overline{X}$.)
Moreover, as $Y$ is convex,  
\[
\partial u_\epsilon(S_\epsilon) \subset \spt (\nu_\epsilon) = \overline{Y} - \epsilon\e_1
\] 
(cf. \eqref{eqn: subdiff props}).
Thus, any open subset of $\partial u_\epsilon(S_\epsilon)$ must be in the interior of the support of $\nu_\epsilon$.
In turn, considering the cone generated by $\partial S_\epsilon$ over $(0,u_\epsilon(0))$, for $0 < \epsilon \ll 1$, we find that
\[
0 = \nu_\epsilon(\partial u_\epsilon(S_\epsilon))  \geq \nu_\epsilon(B_{r_\epsilon})> 0 \quad\text{with}\quad r_\epsilon := \frac{|u_\epsilon(0)|}{2\diam(S_0)}.
\]
(Again, for more details on this inclusion, see, e.g., \cite[Theorem 2.8]{F}.)
This is a contradiction and concludes the proof.
\end{proof}

\section{Proof of Theorem~\ref{thm: C1}}
\label{sec: C1}
Again, we replace $u$ by its lower-semicontinuous extension outside of $X$, exactly as we did at the beginning of Section~\ref{sec: strict convexity}.
We split the proof in three parts.

\subsection*{Part 1: $u$ is continuously differentiable inside $X$.}
We follow the argument used to prove \cite[Corollary 1]{C0}.
Assume for the sake of a contradiction that the result is false.
Up to a translation, let $0 \in X$ be a point at which $u$ has two distinct supporting planes.
After a rotation, dilation, and subtracting off an affine function from $u$, we may assume that
\[
u(x) \geq \max \{ x_1,0 \},\quad u(0) = 0, \quad\text{and}\quad\frac{u(-s\e_1)}{s} \to 0 \quad\text{as}\quad s \to 0.
\]
Now consider the function $u_\sigma$ defined by
\[
u_\sigma(x) := u(x) - \tau( x_1 + 2\sigma) \quad\text{with}\quad \tau := \frac{u(-\sigma\e_1 )}{\sigma}.
\]
Note that $\tau \to 0$ as $\sigma \to 0$.
If
\[
S_\sigma := \{ u_\sigma \leq 0 \},
\]
then, by the strict convexity of $u$  provided by Theorem~\ref{thm: strict convexity}, we see that
\[
S_\sigma \subset B_D \Subset X \cap \interior (\dom(u))
\]
for some $D > 0$ and for all $\sigma \ll 1$; also, for these small positive $\sigma$,
\[
\partial u_\sigma (S_\sigma) \subset B_R \cap Y
\]
for some $R > 0$.
Moreover, if $\Pi_{-a} := \{ x_1 = -a \}$ and $\Pi_{b} := \{ x_1 = b \}$ denote the two parallel planes that tangentially sandwich $S_\sigma$, we see that
\[
a > \sigma\quad  \text{and}\quad b < \frac{2\tau \sigma}{1 - \tau},
\]
provided $\sigma > 0$ is small enough to guarantee that $\tau < 1$.
Furthermore,
\[
\max_{S_\sigma} |u_\sigma| = |u_\sigma(0)| = 2\tau \sigma,
\]
and
\[
\frac{\dist(\Pi_b,\Pi_0)}{\dist(\Pi_{-a},\Pi_b)} = \frac{b}{a+b} \leq \frac{b}{a} \leq \frac{2\tau}{1 - \tau} \to 0 \quad\text{as}\quad \sigma \to 0,
\]
where $\Pi_0 := \{ x_1 = 0\}$.

Now, set 
\[
\tilde{u}_\sigma(x) := \frac{u_\sigma(A_\sigma^{-1}x)}{2\tau\sigma} \quad \text{and}\quad \tilde{S}_\sigma : = A_\sigma(S_\sigma),
\] 
where $A_\sigma$ is the John map that normalizes $S_\sigma$.
Arguing as we did in the proof of Theorem~\ref{thm: strict convexity}, we find the same contradiction as we did in Case 2 when $\tau$ sufficiently small; the only difference is that we consider a slightly different chain of inequalities:
\[
\begin{split}
\tilde{\nu}_\sigma(K_\sigma) \leq \tilde{\nu}_\sigma\big(\partial \tilde{u}_\sigma(\tilde{S}_\sigma)\big) &= \tilde{\mu}_\sigma(\tilde{S}_\sigma) 
\\
&\leq C_\mu \tilde{\mu}_\sigma(\tfrac{1}{2}\tilde{S}_\sigma) = C_\mu\tilde{\nu}_\sigma\big(\partial \tilde{u}_\sigma(\tfrac{1}{2}\tilde{S}_\sigma)\big) \leq C_\mu\tilde{\nu}_\sigma(B_{\frac{1}{r_n}})
\end{split}
\]
(cf. \eqref{eqn: mb4}) where $\tilde{\mu}_\sigma$ and $\tilde{\nu}_\sigma$ are defined so that $(\nabla \tilde{u}_\sigma)_{\#} \tilde{\mu}_\sigma = \tilde{\nu}_\sigma$ and
\[
K_\sigma := \conv\left(B_{r_n} \cup \left\{ \tfrac{r_n+n(1-\tau)}{2\tau}\e_1\right\}\right) \quad \text{with}\quad r_n := \frac{1}{2n^{3/2}}.
\]
This proves that $u$ is differentiable.

By \cite[Lemma A.24]{F}, for example, we know that differentiable convex functions are continuously differentiable.
So we conclude that $u$ is continuously differentiable in $X$.

\subsection*{Part 2: $\nabla u(X)=Y$ when $X$ is convex.}
Because $\nabla u$ is continuous in $X$, its image $Y':=\nabla u(X)$ is an open set of full $\nu$-measure contained inside $Y$.
Also, as the assumptions on $\mu$ and $\nu$ are symmetric, the optimal transport map $\nabla v$ from $\nu$ to $\mu$ is continuous, and $X':=\nabla v(Y)$ is an open set of full $\mu$-measure contained inside $X$. 
Hence, recalling that $\nabla u$ and $\nabla v$ are inverses of each other (see, e.g., \cite[Corollary 2.5.13]{FG}), we conclude that $X'=X$ and $Y'=Y$, as desired.

\subsection*{Part 3: $\nabla u$ is locally H\"older continuous inside $X$.}
Thanks to the strict convexity and $C^1$ regularity of $u$, we can localize the arguments of the proof of \cite[Theorem 1.1]{JS} to obtain the local H\"older continuity of $u$ inside $X$.

More precisely, if $u^\ast$ denotes the Legendre transform of $u$ (see \eqref{eq:u ast}), as in \cite{JS}, one can show that $u^\ast$ satisfies a weak form of Alexandrov's Maximum principle (see \cite[Lemma 3.2]{JS}), from which one deduces the engulfing property for the sections of $u^\ast$ (see \cite[Lemma 3.3]{JS}). 
Iteratively applying this engulfing property, one obtains a  polynomial strict convexity bound for $u^\ast$.
This bound implies the local H\"older continuity of $u$ inside $X$ (see \cite[Proof of Theorem 1.1]{JS}).
We leave the details of this adaptation to the interested reader.

\bigskip
\noindent {\bf Conflict of Interest:} Authors state no conflict of interest.
\bigskip
\\
\noindent {\bf Funding Information:} AF acknowledges the support of the ERC grant No. 721675 ``Regularity and Stability in Partial Differential Equations (RSPDE)'' and of the Lagrange Mathematics and Computation Research Center. YJ was supported in part by NSF grant DMS-1954363.


\end{document}